\documentclass[11pt,,a4paper]{article} 
\author{Dimitri Chikhladze}
\title{Quantum Modules}
\date{}
\usepackage{amsfonts} 
\usepackage{amsmath}
\usepackage[all]{xy} 
\usepackage{amsthm}
\usepackage{graphicx} 

\makeatletter

\newcommand{\pocorner}{\hbox to 10pt{{\vrule height10pt depth0pt width0.5pt}%
    \vbox to 10pt{{\hrule height0.5pt width9.5pt depth0pt}\vfill}}}
\newcommand{\@poexcursion}[1]{\save[]-<0pt+#1,-12pt>*{\pocorner}\restore}
\newcommand{\pbcorner}{\vbox to 0pt{\kern 5pt\hbox to 0pt{\kern 5pt%
      \vbox{{\hrule height0.5pt width9.5pt depth0pt}}%
      {\vrule height10pt depth0pt width0.5pt}\hss}\vss}}
\def\@pbexcursion[#1]{\save[]+DR-<16pt,-8pt>+<#1,0pt>*{\pbcorner}\restore}
\newcommand{\pbexcursion}{\@ifnextchar[\@pbexcursion{\@pbexcursion[0pt]}}

\def\arr@fn{\futurelet\arr@next}
\def\arr@dn{\def\arr@next}

\newtoks\arr@toks
\def\addtoarr@toks#1{\arr@toks=\expandafter{\the\arr@toks#1}}

\def\arr@upperlab{}
\def\arr@lowerlab{}
\def\arr@mods{}
\def\arr@uppermods{}
\def\arr@lowermods{}

\newbox\arr@box
\newdimen\arr@dimen
\newdimen\arr@spacer
\newif\ifarr@fixeddim

\def\arr@updatedimen#1{%
  \ifarr@fixeddim\else\setbox\arr@box=\hbox{$\m@th\scriptstyle #1$}%
  \ifdim\wd\arr@box>\arr@dimen \arr@dimen=\wd\arr@box\fi\fi}

\def\parsearr@@{%
    \ifx\space@\arr@next \expandafter\arr@dn\space{\arr@fn\parsearr@@}%
    \else\ifx ^\arr@next \arr@dn ^##1{\def\arr@upperlab{^{##1}}\arr@fn\parsearr@@}%
    \else\ifx _\arr@next \arr@dn _##1{\def\arr@lowerlab{_{##1}}\arr@fn\parsearr@@}%
    \else\ifx <\arr@next \arr@dn <##1>{\arr@fixeddimtrue\arr@dimen=##1%
      \arr@fn\parsearr@@}
    \else\addAT@\ifx\arr@next \addAT@\arr@dn[##1]{\def\arr@mods{##1}\arr@fn\parsearr@@}%
    \else\ifarr@fixeddim\else\arr@updatedimen{\arr@upperlab}%
        \arr@updatedimen{\arr@lowerlab}\advance\arr@dimen by \arr@spacer\fi%
      \arr@toks={\xy<0pt,0pt>*+{}\ar}%
      \expandafter\addtoarr@toks\expandafter{\arr@mods}%
      \expandafter\addtoarr@toks\expandafter{\arr@upperlab}%
      \expandafter\addtoarr@toks\expandafter{\arr@lowerlab}%
      \addtoarr@toks{<\arr@dimen,0pt>*+{}\endxy}%
      \arr@dn{\the\arr@toks}%
    \fi\fi\fi\fi\fi\arr@next}

\newcommand{\arrow}{%
  \arr@fixeddimfalse\arr@dimen=0pt\def\arr@upperlab{}\def\arr@lowerlab{}%
  \def\arr@mods{}\arr@fn\parsearr@@}

\def\parsearrpair@@{%
    \ifx\space@\arr@next \expandafter\arr@dn\space{\arr@fn\parsearrpair@@}%
    \else\ifx ^\arr@next \arr@dn ^##1{\def\arr@upperlab{^{##1}}\arr@fn\parsearrpair@@}%
    \else\ifx _\arr@next \arr@dn _##1{\def\arr@lowerlab{_{##1}}\arr@fn\parsearrpair@@}%
    \else\ifx <\arr@next \arr@dn <##1>{\arr@fixeddimtrue\arr@dimen=##1%
      \arr@fn\parsearrpair@@}
    \else\addAT@\ifx\arr@next \addAT@\arr@dn[{\arr@fn\parsearrpair@@@}%
    \else\ifarr@fixeddim\else\arr@updatedimen{\arr@upperlab}%
        \arr@updatedimen{\arr@lowerlab}\advance\arr@dimen by \arr@spacer\fi%
      \arr@toks={\xy<0pt,0pt>*+{}="a", <\arr@dimen,0pt>*+{}="b"}%
      \expandafter\addtoarr@toks\expandafter{\addAT@\ar<2pt>}%
      \expandafter\addtoarr@toks\expandafter{\arr@uppermods}%
      \addtoarr@toks{"a";"b"}%
      \expandafter\addtoarr@toks\expandafter{\arr@upperlab}%
      \expandafter\addtoarr@toks\expandafter{\addAT@\ar<-2pt>}%
      \expandafter\addtoarr@toks\expandafter{\arr@lowermods}%
      \addtoarr@toks{"a";"b"}%
      \expandafter\addtoarr@toks\expandafter{\arr@lowerlab}%
      \addtoarr@toks{\endxy}\arr@dn{\the\arr@toks}%
    \fi\fi\fi\fi\fi\arr@next}

\def\parsearrpair@@@{%
  \ifx u\arr@next \arr@dn u##1]{\def\arr@uppermods{##1}\arr@fn\parsearrpair@@}%
  \else\arr@dn l##1]{\def\arr@lowermods{##1}\arr@fn\parsearrpair@@}%
  \fi\arr@next}

\newcommand{\arrowpair}{%
  \arr@fixeddimfalse\arr@dimen=0pt\def\arr@upperlab{}\def\arr@lowerlab{}%
  \def\arr@uppermods{}\def\arr@lowermods{}\arr@fn\parsearrpair@@}

\makeatother

\newdir{ >}{{}*!/-10pt/@{>}}
\newdir{u(}{{}*!/-6pt/@^{(}}
\newdir{d(}{{}*!/-6pt/@_{(}}
\newdir{|>}{%
  !/4.5pt/@{|}*:(1,-.2)@^{>}*:(1,+.2)@_{>}*+@{}}

\newcommand{\epi}{\arrow@[@{->>}]}
\newcommand{\cover}{\arrow@[@{-|>}]}
\newcommand{\spanarr}{\arrow@[|-@{|}]}
\newcommand{\inc}{\arrow@[@{u(->}]}
\newcommand{\mapto}{\arrow@[@{|->}]}
\newcommand{\mono}{\arrow@[@{ >->}]}
\newcommand{\mat}{\arrow@[|-*{\object@{o}}]}
\newcommand{\twocell}{\arrow@[@{=>}]}

\input{diagxy} 
\begin{document}
\newtheorem{theorem}{Theorem}
\newtheorem{lemma}[theorem]{Lemma}
\newtheorem{prop}[theorem]{Proposition}
\newtheorem{corollary}[theorem]{Corollary}
\theoremstyle{definition}
\newtheorem{definition}[theorem]{Definition} 

\maketitle 
\begin{abstract}
There are various generalizations of bialgebras to their ``many object''
versions, such as quantum categories, bialgebroids and weak bialgebras. These
can also be thought of as quantum analogues of small categories.
In this paper we study modules over these structures, which are quantum
analogues of profunctors (also called distributors) between small categories.
\end{abstract}

\section{Introduction} 

Notions of $\times_A$-coalgebra and $\times_A$-bialgebra were
introduced by Takeuchi \cite{Takeuchi}. Takeuchi's $\times_A$-bialgebras
generalize bialgebras and are a special case of quantum categories \cite{DS},
which are defined for an arbitrary braided monoidal category $\mathcal{V}$ and also include small categories. 

In this paper we define modules over quantum categories. Modules over
$\times_A$-bialgebras have been considered before. However our definition is
the natural one from the point of view of category theory. In the $\mathcal{V} =
Set$ case it gives profunctors between small categories. Further, we discuss the
question of composing such modules, analogously to composing profunctors.

First we work in an arbitrary braided monoidal category $\mathcal{V}$. Then we
consider several special cases. In Section 3 we briefly examine the $\mathcal{V}
= Set$ case. The setting of Section 4 is that of Takeuchi \cite{Takeuchi}.
Here we also obtain a result about associativity of the operation $\times_A$.
Section 5 is dedicated to weak bialgebras. Takeuchi's
operation $\times_A$ is computed for weak bialgebras.

\section{Comonads, monoidales and Kan extensions}\label{sec2}
 
In this section we will work with a monoidal bicategory $\mathcal{B}$. We
assume that for every $n > 2$ a choice of an $n$-ary tensor product
pseudofunctor

$$\mathcal{B}^n \to^{\otimes_n} \mathcal{B}$$

\noindent is made, which involves choosing an order of bracketing for the tensor
product. The expression $B_1\otimes\ldots\otimes B_n$ refers to
$\otimes_n(B_1\otimes\ldots\otimes B_n)$.
 
A comonad in $\mathcal{B}$ is a pair $(B, g)$
where $B$ is an object of $\mathcal{B}$ and $g = (g, \delta : g \Rightarrow gg, \epsilon : g \Rightarrow 1_g)$ is a comonoid in the homcategory
$\mathcal{B}(B, B)$. A map $(k, \kappa) : (B, g) \to (B', g')$ of comonads
consists of a 1-cell $k : B \Rightarrow B'$ and a 2-cell $\kappa : kg
\Rightarrow g'k$ satisfying the conditions

$$(kg \to^{k\delta} kgg \to^{\kappa g} g'kg \to^{g'\kappa} g'g'k)
\; = \; (kg \to^{\kappa} g'k \to^{g'\delta} g'g'k)\text{,}$$

$$(kg \to^{k\epsilon} k) \; = \; (kg \to^{\kappa} g'k
\to^{\epsilon k} k)\text{.}$$

\bigskip

A comonad map transformation $\tau : (k, \kappa) \Rightarrow (k',
\kappa') : (A, g) \to (B, g')$ is a 2-cell $\tau : k \Rightarrow k'$ satisfying

$$(kg \to^{\tau g} k'g \to^{\kappa'} g'k') \; = \; (kg \to^{\kappa}
g'k \to^{g\tau} g'k')\text{.}$$

\bigskip

\noindent Comonads in $\mathcal{B}$, comonad maps and comonad map
transformations form a bicategory $\mathrm{Comnd}{\mathcal{B}}$ under the
obvious compositions.

A monoidale (called ``pseudomonoid'' in \cite{DMS}) in $\mathcal{B}$ consists of an
object $E$, morphisms $p : E\otimes E \rightarrow E$ and $j : I \rightarrow B$ called the multiplication
and the unit respectively, and invertible 2-cells $p(p\otimes 1_E)
\Rightarrow p(1_E\otimes p)$, $p(j\otimes 1_E) \Rightarrow 1_E$ and
$p(1_E\otimes j) \Rightarrow 1_E$ satisfying two axioms. A monoidal morphism
between monoidales $(f, \phi_2, \phi_0) : E \rightarrow D$ consists of a
morphism $f : E \rightarrow D$ and 2-cells $\phi_2 : p(f\otimes f) \Rightarrow
fp$, $\phi_0 : j \Rightarrow fj$ satisfying three axioms. Monoidales in
$\mathcal{B}$, monoidal morphisms and the obvious 2-cells form a bicategory $\mathrm{Mon}\mathcal{B}$.

For any monoidal $E$ there is an $n$-ary multiplication map

$$E^n \to^{p_n} E\text{.}$$

\noindent It is defined by consecutive multiplications following the order
of the chosen bracketing for the tensor product in $\mathcal{B}$.

A monoidal comonad is a comonad in $\mathrm{Mon}\mathcal{B}$.
Explicitly, it consists of a monoidale $E$ a comonad $g$
on $E$ such that the comultiplication $\delta : g \Rightarrow gg$ and the counit $\epsilon : g
\Rightarrow 1_g$ are maps of monoidal morphisms.

Suppose that $E$ is a monoidale and $g : E \to E$ is an endomorphism such that
the left Kan extensions $Lan_p(p(g\otimes g))$, $Lan_p(p(Lan_p(p(g\otimes
g))\otimes g))$, $Lan_p(p(g\otimes Lan_p(p(g\otimes g))))$ and
$Lan_{p_3}(p_3(g\otimes g\otimes g))$ exist. Giving a monoidal structure on $g$
is equivalent to giving 2-cells $\mu : Lan_p(p(g\otimes g)) \Rightarrow g$ and
$\eta : Lan_jj \Rightarrow g$ satisfying the conditions

\begin{equation}
\label{comonoid} 
\bfig
\Dtriangle(2050,0)/`>`<-/<500,250>[Lan_p(p(g\otimes g))`g`Lan_p(p(g\otimes
g));`\mu`\mu] 
\square(500,0)/>```>/<1550,500>[Lan_p(p(Lan_p(p(g\otimes g))\otimes
g))`Lan_p(p(g\otimes g))`Lan_p(p(g\otimes Lan_p(p(g\otimes g))))`Lan_p(p(g\otimes
g));Lan_p(p(\mu\otimes g))```Lan_p(p(g\otimes\mu))]
\Ctriangle/<-``>/<500,250>[Lan_p(p(Lan_p(p(g\otimes g))\otimes
g))`Lan_{p_3}(p_3(g\otimes g\otimes g))`Lan_p(p(g\otimes Lan_p(p(g\otimes
g))));``]
\square(500,-1000)/>`>`<-`>/<1350,500>[g`g`Lan_p(p(Lan_jj\otimes
g))`Lan_p(p(g\otimes g));1_g``\mu`Lan_p(p(\eta\otimes g))]
\square(500,-2000)/>`>`<-`>/<1350,500>[g`g`Lan_p(p(g\otimes
Lan_jj))`Lan_p(p(g\otimes g));1_g``\mu`Lan_p(p(g\otimes \eta))]
\efig 
\end{equation} 

\bigskip

\noindent The unnamed
arrows here and below are the canonical maps, determined by the universal
properties of left Kan extensions.

Suppose that $(B, g)$ is a comonad and $k : B \to B'$ is a 1-cell. Assume that
the left Kan extension $Lan_k(kg)$ exists and let $\kappa : kg \Rightarrow
Lan_k(kg)k$ be the universal 2-cell. The pair
$(B', Lan_kkg)$ can be uniquely turned into a comonad so that $(k, \kappa)$
becomes a comonad map \cite{FTM1}. Furthermore, there is a correspondence between comonad maps:

$$
\bfig
\morphism(-50,-250)<800,0>[(B', Lan_kg)`(B', g');(1_{B'},
\kappa')]
\morphism(-300, -100)/-/<1200,0>[`;]
\morphism[(B, g)`(B, g');(k, \kappa)]
\efig
$$

\bigskip

\noindent Or more precisely there is an equivalence of
categories:

\begin{equation}\label{equiv}
\mathrm{Comnd}\mathcal{B}((B, g), (B', g')) \simeq
\mathrm{Comon}\mathcal{B}(Lan_kkg, B')\text{.}
\end{equation} 

\bigskip

\noindent $\mathrm{Comon}$ stands for the category of comonoids.

Suppose that $E$ is a monoidale and $g$ is a comonad on $E$. Using
\eqref{equiv} it can be seen that giving a monoidal structure on the comonad $g$
is equivalent to giving comonoid maps $\mu : Lan_p(g\otimes g) \to g$ and $\eta
: Lan_jj \to g$ such that the diagrams \eqref{comonoid} commute, now in the
category $\mathrm{Comon}\mathcal{B}(E, E)$.

The reader might recognize the appropriateness of the context of multitensor
categories. Provided certain left Kan extensions exist, a monoidale structure on $E$
determines a lax monoidal structure on $\mathcal{B}(E, E)$. The $n$-ary
tensor product is 

$$Lan_{p_n}(p_n(-\otimes\ldots\otimes -))\text{.}$$

\bigskip  

\noindent Notion of the comonoid makes sense in any multitensor category. A
monoid in $\mathcal{B}(E, E)$ is a monoidal endomorphism on $E$.
The multitensor structure of $\mathcal{B}(E, E)$ can be lifted to $\mathrm{Comon}\mathcal{B}(E, E)$. A monoid in
$\mathrm{Comon}\mathcal{B}(E, E)$ is a monoidal comonad on $E$.

\begin{definition}
For monoidales $E$ and $E'$, an $(E, E')$-actee is a pseudoalgebra for the
pseudomonad $E\otimes - \otimes E'$ on $\mathcal{B}$. A map between actees is a
map of pseudoalgebras.
\end{definition}

An $(E, E')$-actee structure on an object $B$ consists of a morphism $a :
E\otimes B\otimes E' \to B$ and isomorphisms $a(1_E\otimes a\otimes1_E') \Rightarrow a(p\otimes 1_B\otimes p)$, $a(j\otimes 1_B\otimes j) \Rightarrow 1_B$ satisfying
the two axioms. Here are two special cases of this concept:

\begin{definition} Suppose that $g : E \to E$ and $g'
: E' \to E'$ are monoidal endomorphisms. A $(g, g')$-action on a endomorphism $m
: B \to B$ consists of a morphism $a : E\otimes B\otimes E' \to B$ and a 2-cell
$\gamma : a(g\otimes m\otimes g') \Rightarrow ma$ satisfying axioms.
\end{definition} 

\begin{definition}Suppose that
$(E, g)$ and $(E', g')$ are monoidal comonads. A $(g, g')$-action on a comonad
$(B, m)$ consists of a morphism $a : E\otimes B\otimes E' \to B$ and a comonad
map of the form $(a, \gamma) : (E\otimes B\otimes E', g\otimes m\otimes g') \to (B, m)$ satisfying axioms.
\end{definition}

In both cases there is an underlying $(E, E')$-action on the object
$B$. With existence of the left Kan extensions, a $(g, g')$-action on $m$, with
a given underlying $(E, E')$-action on $B$, is determined by a 2-cell $\alpha :
Lan_a(a(g\otimes m\otimes g')) \Rightarrow m$ satisfying two axioms:

\begin{align}
\bfig
\Dtriangle(2700,0)/`>`<-/<500,250>[Lan_{a}(a(g\otimes m\otimes 
g'))`m`Lan_a(a(g\otimes m\otimes g'));`\alpha`\alpha] 
\square(500,0)/>```>/<2200,500>[Lan_{a}(a(g\otimes Lan_a(a(g\otimes
m\otimes g'))\otimes g'))`Lan_a(a(g\otimes m\otimes
g'))`Lan_a(a(Lan_p(p(g\otimes g))\otimes m\otimes Lan_p(p(g'\otimes
g'))))`Lan_a(a(g\otimes m\otimes g'));Lan_a(a(g\otimes\alpha\otimes
g'))```Lan_a(a(\mu\otimes m\otimes\mu))]
\Ctriangle/<-``>/<500,250>[Lan_{a}(a(g\otimes Lan_a(a(g\otimes m\otimes
g'))\otimes g'))`Lan_{a_3}(a_3(g\otimes g\otimes m\otimes g'\otimes
g'))`Lan_a(a(Lan_p(p(g\otimes g))\otimes m\otimes Lan_p(a(g'\otimes g'))));``]
\square(500,-1000)/>`>`<-`>/<1750,500>[m`m`Lan_a(a(Lan_jj\otimes m\otimes
Lan_jj))`Lan_a(a(g\otimes m\otimes g'));1_m``\alpha`Lan_a(a(\eta\otimes
m\otimes\eta))]
\efig 
\end{align}

\bigskip 

\noindent In the case of comonads, $\alpha$
should be a comonoid map, and the diagrams above
should commute in $\mathrm{Comon}\mathcal{B}(B, B)$.

For a $(g, g')$-action on $m$, the left action map
$\alpha_l$ is defined to be the composite:
 
$$
\bfig
\morphism<1100,0>[Lan_{a_l}(a_l(g\otimes m))`Lan_a(a(g\otimes
m\otimes Lan_jj));]
\morphism(1100,0)<1700,0>[Lan_a(a(g\otimes
m\otimes Lan_jj))`Lan_a(a(g\otimes
m\otimes g'));Lan_a(a(g\otimes m\otimes \eta))]
\morphism(2800,0)<650,0>[Lan_a(a(g\otimes
m\otimes g'))`m;\alpha]
\efig
$$

\bigskip

\noindent and the right action map $\alpha_r$ is defined to be the composite:

$$
\bfig
\morphism<1100,0>[Lan_{a_r}(a_r(m\otimes g'))`Lan_a(a(Lan_jj\otimes
m\otimes g'));]
\morphism(1100,0)<1700,0>[Lan_a(a(Lan_jj\otimes
m\otimes g'))`Lan_a(a(g\otimes m\otimes
g'));Lan_a(a(g\otimes m\otimes \eta))]
\morphism(2800,0)<650,0>[Lan_a(a(g\otimes m\otimes
g'))`m\text{,};\alpha]
\efig
$$

\bigskip

\noindent where $a_l = a(1_E\otimes j)$ and $a_r = (j\otimes 1_E')$.

\section{Quantum Modules}\label{qm} 
Let $\mathcal{V} = (\mathcal{V}, \otimes, c)$ be a braided monoidal category.
Assume that each of the functors $X\otimes -$ preserves coreflexive equalizers.

We will work with a monoidal bicategory $\mathrm{Comod}\mathcal{V}$
considered in \cite{DMS}, which will be taken as the monoidal bicategory
$\mathcal{B}$ of the previous section. Objects of $\mathrm{Comod}\mathcal{V}$
are the comonoids $C = (C, \delta : C \to C\otimes C, \epsilon : C \to I)$ in $\mathcal{V}$. The
homcategory $\mathrm{Comod}\mathcal{V}(C, D)$ is the category of Eilenberg-Moore
coalgebras for the comonad $C\otimes - \otimes D$. A 1-cell from $C$ to $D$,
depicted as $C \spanarr D$, is a comodule from $C$ to $D$. The
composition $N\circ M$ is defined by a coreflexive equalizer

$$M\otimes_C N\to M\otimes N \two^{\delta_r\otimes 1}_{1\otimes \delta_l} M\otimes C\otimes N\text{.}$$

\bigskip

\noindent The monoidal structure of $\mathrm{Comod}\mathcal{V}$ extends that of
$\mathcal{V}$ (although it is not braided). Each comonoid $C = (C, \delta,
\epsilon)$ has an opposite comonoid $C^o = (C, c\delta, \epsilon)$. There are comodules 

$$e : C\otimes C^o \spanarr I \qquad n : I \spanarr C^o\otimes C\text{,}$$

\bigskip

\noindent both of which are $C$ as objects with coactions in string notation
respectively:

\begin{center}
\includegraphics[scale=0.3]{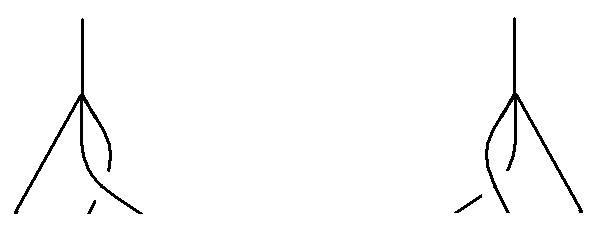}
\end{center}

\noindent This exhibits $C^o$ as a right bidual of $C$ in
$\mathrm{Comod}\mathcal{V}$. It follows that $C^o\otimes C$ is a monoidale
with multiplication $p = C^o\otimes e\otimes C$ and unit $j = n$.

Throughout this paper a right $C^o\otimes C'$-comodule
$X$ will be regarded as a comodule $C \spanarr C'$ using biduality when tensor
products $X\otimes_{C'} -$ or $-\otimes_C X$ are taken.

Let $X_1 : I \spanarr A\otimes C^o$ and $X_2 : I \spanarr C\otimes B$
be comodules. With little calculation it can be established that the composite
comodule

$$I \spanarr^{X_1\otimes X_2} A\otimes C^o\otimes C \otimes B
\spanarr^{A\otimes e\otimes B} A\otimes B$$

\noindent is $X_1\otimes X_2$ with the right $A$- and $B$-coactions on it
induced from the right $A$-coaction on $X_1$ and the right $B$-coaction on
$X_2$.

\begin{definition}
An algebroid in $\mathcal{V}$ is a pair $(A, C)$, where $C$ is
a comonoid in $\mathcal{V}$ and $A$ is a monoidal endomorphism on $C^o\otimes C$. 
\end{definition} 

\begin{definition} (see \cite{DS})
A quantum
category in $\mathcal{V}$  is a pair $(A, C)$, where $C$ is a comonoid in $\mathcal{V}$ and $A$
is a monoidal comonad on $C^o\otimes C$.
\end{definition}

For comonoids $C$ and $C'$, the map $a = C\otimes
e\otimes e\otimes C'$ determines a $(C^o\otimes C, C'^o\otimes C')$-action
on $C^o\otimes C'$.

\begin{definition} A module from an algebroid $(A, C)$ to an algebroid $(A',
C')$ consists of a comodule $M : C'^o\otimes C \spanarr C'^o\otimes C$ and a $(A, A')$-action on
$M$, such that the underlying $(C^o\otimes C, C'^o\otimes C')$-action on
$C^o\otimes C$ is the canonical one.
\end{definition}

\begin{definition}
A quantum module from a quantum category $(A, C)$ to a quantum category
$(A', C')$ consists of a comonad $(M, C'^o\otimes C)$ and an $(A, A')$-action on
$(M, C'^o\otimes C)$, such that the underlying $(C^o\otimes C, C'^o\otimes C')$
action on $C^o\otimes C'$ is the canonical one.
\end{definition}

A (quantum) module $M$ from $(A, C)$ to $(A', C')$ has a coaction 2-cell:

$$
\bfig
\square/@{>}|-*@{|}`@{>}|-*@{|}`@{>}|-*@{|}`@{>}|-*@{|}/<1700,600>[C^o\otimes
C\otimes C^o\otimes C'\otimes C'^o\otimes C'`C^o\otimes
C\otimes C^o\otimes C'\otimes C'^o\otimes C'`C^o\otimes C'`C^o\otimes
C';A\otimes M\otimes A'` C^o\otimes e\otimes e\otimes C'`C^o\otimes
e\otimes e\otimes C'`M] 
\morphism(850,350)/=>/<0,-100>[`;\gamma]
\efig
$$

\noindent satisfying two axioms. In the case of a quantum module
$(C^o\otimes e\otimes e\otimes C', \gamma)$ should be a comonad map.

A map of (quantum) modules is a comodule map $M_1 \to M_2$ respecting
the action (for quantum modules it also should be a comonad transformation).

We will apply the machinery of Section \ref{sec2} to our present context. For
this we will need existence of certain left Kan extensions in
$\mathrm{Comod}\mathcal{V}$, and that will be discussed prior. First we
introduce the following structure on the class of comodules of the form $X : C^o\otimes C' \spanarr C^o\otimes C'$ (strictly speaking on the class of triples $(X, C, C')$,
where $C$ and $C'$ are comonoids and $X$ is a comodule of the indicated form).

For comodules $X_i : C_{i-1}^o\otimes C_{i} \spanarr C_{i-1}^o\otimes C_{i}, 1
\leq i \leq n$, define $T_{(C_0, C_2,\ldots C_{n})}(X_1, X_2 \ldots X_n)$ or
simply $T_n(X_1, X_2 \ldots X_n)$ to be the comodule determined by the left Kan
extension
 
\begin{equation}
\label{Tn}
\bfig
\square/@{>}|-*@{|}`@{>}|-*@{|}`@{>}|-*@{|}`@{>}|-*@{|}/<1700,600>[C^o_0\otimes
C_1\otimes\ldots\otimes C_{n-1}\otimes C_n`C^o_0\otimes
C_1\otimes\ldots\otimes C_{n-1}\otimes C_n`C^o_0\otimes C_n`C^o_0\otimes
C_n;X_1\otimes\ldots\otimes X_n`C^o\otimes e\otimes\ldots\otimes e\otimes
C_n`C^o\otimes e\otimes\ldots \otimes e\otimes
C_n`T_n(X_1, \ldots X_n)] 
\morphism(850,350)/=>/<0,-100>[`;]
\efig
\end{equation}

\bigskip

\noindent when this exists. For $n = 1$ this gives $T_1(X_1) = X_1$. For a
comonoid $C$ define $T_C()$ or simply $T_0()$ to be the comodule determined by the left Kan extension

$$
\bfig
\Atriangle/@{>}|-*@{|}`@{>}|-*@{|}`@{>}|-*@{|}/[I`C^o\otimes C`C^o\otimes
C;n`n`T_0()]
\morphism(500,300)/=>/<0,-100>[`;]
\efig
$$

\noindent when this exists. Clearly, the $T_n$ can be made into
functors.

For each partition $\xi :
m=m_1+m_2+\ldots+m_n$, $m_i \geq 0$, the universal properties of left Kan
extensions give an associativity map:

\begin{equation}
\label{alpha}
\beta_{\xi} : T_m(X_{11}, \ldots, X_{nm_m}) \to
T_n(T_{m_1}(X_{11}, \ldots, X_{1m_1}), \ldots, T_{m_n}(X_{n1}, \ldots,
X_{nm_n}))\text{.}
\end{equation} 

\bigskip These are natural in all variables and satisfy coherence
conditions.

When it exists, let $\mathrm{coHom}(X, Y)$ be the internal cohom object in
$\mathcal{V}$, meaning that there is a natural bijection:

$$
\bfig
\morphism(100,-250)[\mathrm{coHom}(X, Y)`Z;f^\ast]
\morphism(-300, -100)/-/<1000,0>[`;]
\morphism(-100,0)[Y`X\otimes Z;f]
\efig
$$

\bigskip

\noindent If $X$ and $Y$ are left $C$-comodules, then $\mathrm{coHom}_C(X, Y)$
is defined by the coequalizer

$$\mathrm{coHom}(C\otimes X, Y) \two^{(coev_{X
,Y}\delta_l^X)^\ast}_{(\delta_l^Ycoev_{X, Y})^\ast} \mathrm{coHom}(X, Y) \to
\mathrm{coHom}_C(X, Y)\text{.}$$

\bigskip

\noindent If $X : C \spanarr A$ and $Y :
C \spanarr B$ are comodules, then $\mathrm{coHom}_C(X, Y)$ becomes a
$A \spanarr B$ comodule. The left Kan extension of $Y$ along $X$ is $\mathrm{coHom}(X,
Y)$.

We deduce that the left Kan extensions \eqref{Tn} exist if $\mathrm{coHom}(C,
X)$ exists for every $X$. 

For $n=2$, $T_2(X_1, X_2)$ can be computed as (setting for simplicity of
notation $C_1 = C$):

\begin{equation*} 
\mathrm{coHom}_{(C^o\otimes C\otimes C^o\otimes C)}(C\otimes
C\otimes C, (C_0\otimes e\otimes C_2)\circ(X_1\otimes X_2)) \cong
\end{equation*}

\begin{equation*} 
\mathrm{coHom}_{(C^o\otimes C\otimes C^o\otimes C)}(C\otimes
C\otimes C, X_1\otimes_C X_2) \cong
\end{equation*}

\begin{equation}\label{ccoc} 
\mathrm{coHom}_{(C\otimes C^o)}(C, X_1\otimes_C X_2)\text{.}
\end{equation}

\bigskip

\noindent It can be shown that this is isomorphic to
the coequalizer of the pair

\begin{equation}
\label{coeqpair}
\mathrm{coHom}(C, X_1\otimes_CX_2)
\two^{(\delta_l^C)^\ast}_{(\delta_l^{C^o})^\ast} X_1\otimes_CX_2\text{,}
\end{equation}

\bigskip 

\noindent wherein the left $C$-coaction $\delta_l^C$ on $X_1\otimes X_2$ is
induced by the left coaction of $C$ on $X_1$, and the left $C^o$-coaction
$\delta_l^{C^o}$ on $X_1\otimes X_2$ is induced by the left coaction of $C^o$ on
$X_2$.

For $n = 0$ we have

$$T_C() = \mathrm{coHom}(C, C)\text{.}$$

The next lemma provides an even more general situation when the
operations $T_n$ can be defined. We need the following definition. 

A $C\otimes$-coequalizer of the pair of morphisms in $\mathcal{V}$

$$Y \two^f_g C\otimes X$$

\bigskip 

\noindent is a map $h : X \to Z$ for which $(1\otimes h)f = (1\otimes h)g$
such that for any other map $h' : X \to Z'$ for which $(1\otimes h')f =
(1\otimes h')g$ there exists a unique $z : Z \to Z'$ with $zh = h'$.

\begin{lemma}\label{lemma}
Suppose that $Y$ is a $A\otimes C\otimes 
C^o\otimes B \spanarr D$ comodule. Suppose that $Z$ is the
$C\otimes$-coequalizer of the pair

\begin{equation}
\label{L}
Y \two^{\delta^C_l}_{\delta^{C^o}_l} C\otimes Y\text{.}
\end{equation}

\bigskip

\noindent Then $Z$ becomes a comodule
$A\otimes B \spanarr D$. The left Kan extension of $Y$
along $A\otimes e\otimes B$ is $Z$.

$$
\bfig
\Atriangle[A\otimes C^o\otimes C\otimes B`A\otimes B` D;A\otimes
e\otimes B`Y`Z]
\morphism(500,300)/=>/<0,-100>[`;]
\efig
$$

\end{lemma}

\begin{proof}
Let $h : Y \to Z$ be the $C\otimes$-coequalizer of \eqref{L}. In the diagram

$$
\bfig
\square/`>`>`/<650,500>[Y`C\otimes Y`A\otimes Y`A\otimes C\otimes
Y;`\delta^A_l`c^{-1}\delta_l^{A}`]
\morphism(150,20)<230,0>[`;A\otimes\delta^C_l]
\morphism(150,-20)|b|<230,0>[`;A\otimes\delta^{C^o}_l]
\morphism(80,520)<400,0>[`;\delta^C_l]
\morphism(80,480)|b|<400,0>[`;\delta^{C^o}_l]
\square(650,0)/>``>`>/<900,500>[C\otimes Y`C\otimes Z`A\otimes C\otimes
Y`C\otimes A\otimes Z;1\otimes h``1\otimes\delta_l^A`(c\otimes1)(1\otimes
1\otimes h)]
\efig
$$

\noindent the top and the bottom horizontal parts are
commutative and so are the two squares to the left. Hence there exists a unique
map $\delta^A_l: Z \to A\otimes Z$ rendering the right square commutative. This
map defines left $A$-coaction on $Z$. Left $B$- and right $D$-coactions on $Z$
are defined similarly. Thus $Z$ becomes a comodule $A\otimes B \spanarr D$.

For a comodule $X : A\otimes B \spanarr D$ the composite comodule
$X\circ(1\otimes e\otimes1)$ is $C\otimes X$. The
left $A$- and $B$- and right $D$-coactions on $C\otimes X$ are induced by the respective coactions on
$X$. The left $C$- and $C^o$-coactions on $X$ are the cofree coactions, meaning
that they are determined by comultiplications. Using this fact we can
establish that a 2-cell

$$Y \Rightarrow X\circ(1\otimes e\otimes1)$$

\noindent is a map $h' : Y \to C\otimes X$ which respects left $A$- and
$B$- and right $D$-coactions and satisfies

$$Y \two^{\delta_l^C}_{\delta_l^{C^o}} C\otimes Y \to^{1\otimes h'} C\otimes X
\text{.}$$
 
Define the univeral 2-cell

$$Y \Rightarrow Z\circ(1\otimes e\otimes1)$$ 

\noindent to be the map $(1\otimes h)\delta^C_l : Y \to C\otimes Z$. The
univeral property follows from the above and the definition of
$C\otimes$-coequalizer.
\end{proof}

Taking $A = C_0^o, C = C_1, B = C_2, 
D = C_0^o\otimes C_2$ and $X = p\circ(X_1\otimes X_2) = X_1\otimes_C X_2$ we
get $T_2(X_1, X_2)$ to be the $C\otimes$-coequalizer of

\begin{equation}
\label{pair}
X_1\otimes_CX_2 \two^{\delta_l^{C}}_{\delta_l^{C^o}} C\otimes
(X_1\otimes_CX_2)\text{.}
\end{equation}  
 
\noindent (as before we have rendered $C_1 = C$). If the internal cohom exists,
then we can transpose $C$ to the left and that will get us exactly the
coequalizer diagram \eqref{coeqpair}.

For $n \geq 2$, we can write $C^o_0\otimes e\otimes\ldots\otimes e
\otimes C_n$ as a composite

$$(C^o\otimes e \otimes C_n)\circ(C^o\otimes e \otimes
C_{n-1}\otimes C^o_{n-1}\otimes C_n)\circ\ldots\circ(C^o_0\otimes e
\otimes C_2\otimes\ldots\otimes C_n)$$

\bigskip 

\noindent The left Kan extensions along $C^o_0\otimes e\otimes\ldots\otimes
e \otimes C_n$ can be computed by consecutive applications of Lemma
\ref{lemma}. In particular $T_n$ for $n>2$ can be computed in this way.

Assume henceforth that the operations $T_n$ are defined. From Section \ref{sec2}
we obtain the following alternative definitions of algebroids, quantum
categories and their modules: 

An algebroid $(A, C)$ in $\mathcal{V}$ consists of a comodule $A : C^o\otimes C
\spanarr C^o\otimes C$ together with maps $\mu : T_2(A, A) \to A$ and $\eta
: T_0() \to A$ satisfying three conditions.

A module $M$ from an algebroid $(A, C)$ to an algebroid $(A', C')$ consists of a
comodule $M : C^o\otimes C' \spanarr C^o\otimes C'$ and a map $\alpha : T_3(A,
M, A') \to M$ satisfying two conditions.

A quantum category $(A, C)$ in
$\mathcal{V}$ consists of a comonad $A$ on $C^o\otimes C$ together with
comonoid maps $\mu : T_2(A, A) \to A$ and $\eta : T_0() \to A$ satisfying three
conditions.

A quantum module $M$ from a quantum category $(A, C)$ to a quantum category
$(A', C')$ consists of a comonad $M$ on $C^o\otimes C$ and a comonoid
map $\alpha :T_3(A, M, A') \to M$ satisfying three conditions.

For a (quantum) module the left action map $\alpha_l$ and the right action map
$\alpha_r$ are given as:

$$\alpha_l = \Big(T_2(A, M) \to^{0+2} T_3(A, M, T_0()) \to^{T_3(A, M, \eta)}
T_3(A, M, A') \to^{\alpha} M\Big)$$

$$\text{and} \qquad \alpha_r = \Big(T_2(M, A) \to^{2+2} T_3(T_0(), A, M)
\to^{T_3(\eta, M, A')} T_3(A, M, A') \to^{\alpha} M\Big)$$

Suppose that $(A_1, C_1), (A_2, C_2)$ and $(A_3, C_3)$ are algebroids. Let $M_1$
be a module from $(A_1, C_1)$ to $(A_2, C_2)$ and let $M_2$ be a module from
$(A_2, C_2)$ to $(A_3, C_3)$. Define $M\bullet N$ by the coequalizer  

\begin{equation}\label{bullet}
T_3(M_1, A_2, M_2) \two^{T_2(M_1,\alpha_l)\beta_{1+2}}_{T_2(\alpha_r,
M_2)\beta_{2+1}} T_2(M_1, M_2) \to M_1\bullet N_2
\end{equation}

\bigskip

\noindent in $\mathrm{Comod}\mathcal{V}(C_1^o\otimes C_3, C_1^o\otimes C_3)$.
Coequalizers in the comodule category are computed as in $\mathcal{V}$.

Generally the opperation $\bullet$ is not associative (which is not surprising
since $T_2$ itself is not associative). It is not even a proper
composition since $M_1\bullet M_2$ does not become a module from
$(A_1, C_1)$ to $(A_3, C_3)$. However it does have left and right units.

Given an algebroid $(A, C)$, via the algebroid multiplication, $A$ becomes a
module from $(A, C)$ to $(A, C)$. So the action map is $\mu : T_2(A, A)
\to A$.

\begin{lemma}\label{lu}
We have: $A\bullet M = M$ and $M\bullet A = M$.
\end{lemma}

\begin{proof} 
The diagram 

$$T_3(A, A, M) \two^{T_2(1, \alpha_l)\beta_{1+2}}_{T_2(\mu, 1)\beta_{2+1}}
T_2(A, M) \to^{\alpha_l} 1$$

\noindent is a split coequalizer diagram split by the maps:

$$T_3(A, A, M) \to/<-/^{T_3(\eta, 1, 1)\beta_{0+1+1}} T_2(A, M)
\qquad\text{and}\qquad T_2(A, M) \to/<-/^{T_2(\eta, 1)\beta_{0+1}} M\text{.}$$

\noindent This follows from the calculations below. Aside from algebroid and
module axioms we use the naturality and coherence of the maps $\beta_\xi$. 
 
$$\alpha_lT_2(\eta, 1)\beta_{0+1} = 1\text{;}$$

$$T_2(\mu, 1)\beta_{2+1}T_3(\eta, 1, 1)\beta_{0+1+1} = T_2(\mu,
1)T_2(T_2(\eta, 1), 1)\beta_{2+1}\beta_{0+1+1}$$
$$= T_2(\mu, T_2(\eta,
1), 1)T_2(\beta_{0+1}, 1) = 1;$$

$$T_2(1, \alpha_l)\beta_{1+2}T_3(\eta, A, M)\beta_{0+1+1} = T_2(1,
\alpha_l)T_2(\eta, T_2(1, 1))\beta_{1+2}\beta_{0+1+1}$$
$$ = T_2(\eta, \alpha_l)\beta_{0+2} = T_2(\eta, 1)T_2(1,
\alpha_l)\beta_{0+2} T_2(\eta, 1)\beta_{0+2}\alpha_l$$ 
$$= T_2(\eta,
1)\beta_{0+1}\alpha_l$$

\noindent At the end of the last calculation we used the fact that $\beta_{0+1}
= \beta_{0+2}$. This follows directly from the definitions of
$\beta_\xi$.

We have proved that $A\bullet M = M$. The proof of $A\bullet M = M$ is similar. 

\end{proof}

To make $\bullet$ into an associative composition we need
to restrict the class of algebroids and modules that we consider.

Let $\mathcal{X}$ be a class of comodules of the form $X :
C^{o}\otimes C' \spanarr C^{o}\otimes C'$ such that

\begin{enumerate}\label{conditions}
  \item If $X_1, \ldots, X_n$ are in $\mathcal{X}$, then the left Kan
  extension \eqref{Tn} exists.
  \item If $X_{11}, \ldots, X_{mn}$ are in $\mathcal{X}$, then the map
  $\beta_\xi$ \eqref{alpha} is an isomorphism for any partition $\xi =
  m_1+\ldots+m_n$ with $m_i > 0$.
  \item If $X$ is in $\mathcal{X}$, then the functors 
  
  $$X\otimes_C - : \mathrm{Comod}\mathcal{V}(C, I) \to \mathcal{V}$$
  
  $$-\otimes_C X : \mathrm{Comod}\mathcal{V}(I, C) \to \mathcal{V}$$
    
preserve reflexive coequalizers. 
  \item If $X$ and $Y$ are in $\mathcal{X}$, then so is $X\otimes_CY$.
  \item If $X$ is in $\mathcal{X}$ and $X \to Y$ is an epimorphism, then $Y$
  is in $\mathcal{X}$.
\end{enumerate}

\begin{theorem}\label{t1}
Fix a class $\mathcal{X}$ as above. Consider those algebroids and those
modules between them for which the underlying $C'^o\otimes C \spanarr
C'^o\otimes C$ comodules are in $\mathcal{X}$. These form a bicategory under the composition
$\bullet$.
\label{theorem}
   
\end{theorem} 

\begin{proof} 
The functor $T_2(X, -)$ can be written as
a composition of $X\otimes_C -$ and $Lan_p-$. $Lan_p-$ preserves coequalizers since it is a left adjoint and
$X\otimes_C -$ preserves reflexive coequalizers by condition 3. We deduce that
if $X$ is in $\mathcal{X}$, then $T_2(X, -)$ preserves reflexive coequalizers.
Similarly, $T_2(-, X)$ preserves reflexive coequalizers. So, if $A_1$ and $A_3$
are in $\mathcal{X}$, by the usual argument $M\bullet N$ can be made into a module from $(A_1, C_1)$ to $(A_3, C_3)$.
This works even when $\beta_\xi$ are not isomorphisms, although this may not be
evident. However given the condition of the theorem we can as well assume that
$\beta_xi$ are isomorphisms.

The role of 2-cells in our bicategory are played by module maps. The operation
$\bullet$ naturally extends to module maps giving the horizontal composition of
2-cells.

Under the condition 2, $T_2$ is associative up to coherent
isomorphisms. Then $\bullet$ is also associative up to coherent isomorphisms,
and these isomorphisms are module maps. 

The unit 1-cells are provided by Lemma \ref{lu}.

To get a
bicategory we only need to show that $M_1\bullet M_2$ is in $\mathcal{X}$
provided $M_1$ and $M_2$ are in $\mathcal{X}$. This is guaranteed by conditions
4 and 5 since is $M_1\bullet M_2$ is a quotient of $T_2(M_1, M_2)$ which
itself is a quotient of $M_1\otimes_C M_2$.

\end{proof} 

The operation $\bullet$ can be lifted to quantum modules between quantum
categories by considering the coequalizer \eqref{bullet} in
$\mathrm{ComonComod}\mathcal{V}(C_1^o\otimes C_3, C_1^o\otimes C_3)$.
Coequalizers in the latter are again computed as in $\mathcal{V}$. We
have:

\begin{theorem}\label{t2}
Let $\mathcal{X}$ be as above. Consider those quantum categories
and those quantum modules between them for which the underlying
comodules $C^o\otimes C' \spanarr C^o\otimes C'$ are in $\mathcal{X}$.
These form a bicategory under the composition $\bullet$.
\end{theorem}

\section{The $\mathrm{Set}$ case}

We take $\mathcal{V}$ to be $Set$ with the
monoidal structure the cartesian product. Then, as ponted out for example in
\cite{DS}, $\mathrm{Comod}\mathcal{V} = \mathrm{Span}$. A comodule $X :
C^o\otimes C' \spanarr C^o\otimes C'$ is a span of the form:

\begin{equation}
\label{span}
\bfig
\Atriangle/>`>`/[X`C\times C'`C\times C';(t', s')`(t, s)`]
\efig
\end{equation}

\bigskip

A comonad structure on a span like this is the property $t' = t$ and $s' = s$. 

\noindent The diagram \eqref{pair} becomes

\begin{equation}
\label{setpair}
X_1\times_C X_2 \two^{t'_1pr_1}_{s'_2pr_2} C\times(X_1\times_CX_2)
\end{equation}

\bigskip

\noindent where $X_1\times_C X_2$ is the pullback of

$$X_1 \to^{t_1} C \to/<-/^{s_2} X_2$$ 

\bigskip

If $X_1$ and $X_2$ are comonads, then $t'pr_1 = tpr_1 = spr_2 = s'pr_2$. Thus,
in this case the two parallel arrows in \eqref{setpair} are equal. Then the
$C\times$-coequalizer exists and is $X_1\times_C X_2$ itself. It follows that
for the comonads spans the operations $T_n$, $n>2$, are defined and given by

$$T_n(X_1\ldots X_n) = X_1\times_{C_1}\ldots \times_{C_n-1}X_n\text{.}$$ 

\bigskip

\noindent $T_C() = C$ which is a span $C^o\otimes C \spanarr C^o\otimes C$ with
both legs the diagonal maps. The maps $\beta_\xi$ $\eqref{alpha}$ are obviously
isomorphisms. The functors $- \times_C X$ and $X\times_C -$ preserve coreflexive
equalizers since $Set$ is a locally closed category. Consequently the class of comonad spans satisfies all the conditions of Theorem \ref{t1}.

Quantum categories in $Set$
are the ordinary small categories \cite{DS}. Quantum modules are the
profunctors. The operation $\bullet$ coincides with the usual composition of
profunctors. The bicategory of the Theorem \ref{t1} is $\mathrm{Prof}$.

The category $Set$ can be replaced with any locally
closed finitely complete category. In this case quantum categories will be the
internal categories and the quantum modules will be the internal profunctors
\cite{KW}.

\section{Comodules of bialgebroids}  
In this section we consider our theory for
$\mathcal{V}=(k\operatorname{-Mod})^{\mathrm{op}}$ where $k$ is a commutative
ring. Note that in this case limits in $\mathcal{V}$ are the
colimits in $k\operatorname{-Mod}$, the cohom objects in $\mathcal{V}$ are the hom
objects in $k\operatorname{-Mod}$ and so on. The nomencluture is dual to that of
Section \ref{qm}. Nevertheless, we will freely refer to Section \ref{qm}, so
the reader should be somewhat careful.

The objects of $\mathrm{Comod}(k\operatorname{-Mod})^{\mathrm{op}}$ are the
$k$-algebras $R$, morphisms are the two sided modules between $k$-algebras. The
category $k\operatorname{-Mod}$ is closed, so right Kan extensions exist in
$\mathrm{Comod}(k\operatorname{-Mod})^{\mathrm{op}}$.

The operation $T_2$ is exactly the product $\times_R$ of Takeuchi
\cite{Takeuchi}. By \eqref{ccoc} it is equal to

$$\mathrm{Hom}_{(R^o\otimes R)}(R, X\otimes_R Y)\text{.}$$

\noindent It can be also computed using \eqref{pair} to yield:

$$X\times_R Y = \bigg\lbrace \sum_i m_i\otimes_R n_i \in M\otimes_R N : \sum_i
(x\otimes 1)m_i\otimes_R n_i = \sum_i
m_i\otimes_R (x\otimes 1)n_i \quad \forall x \in R \bigg\rbrace\text{.}$$

\noindent For $n = 0$:

$$T_R() = \mathrm{Hom}(R, R)\text{.}$$

\bigskip
 
The ternary operation of Takeuchi $(-\times_R-\times_R-)$ is a special case of a
slightly more general $(-\times_R-\times_{S}-)$, which is our $T_3$.
Takeuchi's maps

$$\alpha : (X\times_R Y)\times_R Z \to X\times_RY\times_RZ$$

$$\alpha' : X\times_R (Y\times_RZ) \to X\times_RY\times_RZ$$

\bigskip

\noindent are nothing but our $\beta_{2+1}$ and $\beta_{1+2}$. Generally we
set 

$$T_{(R_1, \ldots, R_{n-1})}(X_1\ldots X_{n}) =
X_1\times_{R_1}\ldots\times_{R_{n-1}} X_n\text{.}$$

\bigskip

For a right $T$ module we write $X^T$. Given a module 
$X : R^o\otimes R' \spanarr R^o\otimes R'$, when tensor products
$-\otimes_RX$ and $X\otimes_{R'}-$ are taken, $X$ is regarded as a left $R$-module and a right
$R'$-module by the right $R^o\otimes R'$ action, thus as $X^{R^o}$ and $X^{R'}$.
In contrast, when $X$ appears in homs the left $R^o\otimes R'$-action is used. 
 
\begin{lemma}
If $X^T$ and $Y^S$ are projective modules, then $(X\otimes_T Y)^S$ is a
projective module.
\end{lemma}

\begin{proof}
This follows from the fact that if $Y^S$ is projective then the functor 

$$(-\otimes_T Y)^{S} : \mathrm{Mod}(k, T) \to \mathrm{Mod}(T, S)$$

\bigskip

\noindent preserves projective objects since it is a left adjoint to an
epi-preserving functor $\mathrm{Hom}_{S^o}(Y, -)^T$.
\end{proof}

As an immediate consequence we have:

\begin{lemma}\label{l1} If $X_{i}^{R_i}$ are projective modules, then
$(X_1\otimes_{R_1}\ldots\otimes_{R_{n-1}} X_n)^{R_n}$ is a projective module.
\end{lemma}

We say that a right (left) module is a union of projectives if it is
union of all of its projective submodules. 

\begin{lemma}
\label{l}
If $Y^{T^o}$ is flat and $X^T$ and $Y^{S}$ are unions of projectives, then
$(X\otimes_R Y)^{S}$ is a union of projectives. 
\end{lemma}

\begin{proof}
We can write $X=colimX_i$ and $Y=colimY_j$, where $X_i^T$ and $Y_i^S$ are
projective modules and the colimits are taken over filtered diagrams
whose arrows are injections. We have 

$$X \otimes_T Y = colimX_i \otimes_T Y =
colimX_i \otimes_T Y_j$$ 

\bigskip

\noindent The latter colimit is over a filtered diagram whose arrows are
injections again since $Y^{T^o}$ and $X_i^T$ are flat. Then $(X\otimes_R Y)^S$
is a union of projectives since each of $(X_i\otimes_T Y_j)^S$ is projective by
the previous lemma.
\end{proof}

\begin{lemma}\label{l2} If the $X_{i}^{R_i}$ are unions of
projectives and the $X_{i}^{R_i^o}$ are flat, then the right module
$(X_1\otimes_{R_1}\ldots\otimes_{R_{n-1}} X_n)^{R_n}$ is a union of projectives.
\end{lemma}

The next two lemmas are slight modifications of Lemma \ref{l1} and Lemma
\ref{l2} and their proofs are similar.

\begin{lemma}\label{ll1} If the $X_{i}^{R^o_i}$ are projective modules, then $(X_1\otimes_{R_1}\ldots\otimes_{R_{n-1}} X_n)^{R^o_n}$ is a
projective module.
\end{lemma} 

\begin{lemma} If the $X_{i}^{R^o_i}$ are unions of
projectives and the $X_{i}^{R_i}$ are flat, then the right module
$(X_1\otimes_{R_1}\ldots\otimes_{R_{n-1}} X_n)^{R^o_n}$ is a union of
projectives.
\end{lemma}

Recall from ring theory that a ring $T$ is called right hereditary if any
submodule of a projective right module over $T$ is again projective. $T$ is
called hereditary if both $T$ and $T^o$ are right hereditary.

\begin{lemma}\label{ll}
Every submodule of a union of projectives over a hereditary
ring is a union of projectives.
\end{lemma}

\begin{proof}
Obvious.
\end{proof}

Various conditions under which $\beta_{2+1}$ and $\beta_{1+2}$ are
bijective were given in \cite{Takeuchi} and \cite{Sweedler}. We will
obtain a similar result for an arbitrary partition $\xi=m_1+\ldots m_n$ with
$m_i > 0$.

\begin{theorem}
For any partition $\xi=m_1+\ldots m_n$ with
$m_i > 0$, the map $\beta_\xi$ $\eqref{alpha}$ is an isomorphism, if the base
rings are hereditry and each of $X_{ij}$ is a union of projectives both as a
left module and a right module.
\end{theorem}

\begin{proof}  
Suppose that $S$ and $T$ are rings and $A : S \spanarr I$, $Y : S \spanarr T$
and $Z : I \spanarr T^o$ are modules. There is a natural map:

$$\mathrm{Hom}_{S}(A, Y\otimes_{T}Z) \to \mathrm{Hom}_S(A,
Y)\otimes_{T}Z$$ 

\noindent It can be easily seen that if $A$ is a finitely generated left $S$
module and $Z^{{T}^{op}}$ is projective then this map is an isomorphism.
Also, if $A$ is a finitely generated left $S$-module and $Y^T$ is projective
then there is an isomorphism:

$$\mathrm{Hom}_{S}(A, Y\otimes_{T}Z) \to Y\otimes_{T}\mathrm{Hom}_S(A,
Y)$$

\noindent Suppose that $S^i$ are rings for $i = 1\ldots n$. If for each $i$,
$A_{i}$ is a finitely generated left $S_i$-module and $L_{i}^{T_i}$ and
$\mathrm{Hom}_{A_i}(A_i, L_i))^{T_{i-1}^o}$ are projective, then we have:

$$
\mathrm{Hom}_{S_1\otimes\ldots\otimes S_n}(A_1\otimes\ldots\otimes A_n,
L_1\otimes_{T_1}\ldots\otimes_{T_{n-1}} L_n) \cong $$

$$\mathrm{Hom}_{S_1\otimes\ldots\otimes S_{n-1}}(A_1\otimes\ldots\otimes A_{n-1},
\mathrm{Hom}_{A_n}(A_n, L_1\otimes_{T_1}\ldots\otimes_{T_{n-1}} L_n)) \cong $$

$$\mathrm{Hom}_{S_1\otimes\ldots\otimes S_{n-1}}(A_1\otimes\ldots\otimes A_{n-1},
L_1\otimes_{T_1}\ldots\otimes_{T_{n-1}} \mathrm{Hom}_{A_n}(A_n, L_n)) \cong $$

$$\mathrm{Hom}_{S_1\otimes\ldots\otimes S_{n-1}}(A_1\otimes\ldots\otimes A_{n-1},
L_1\otimes_{T_1}\ldots\otimes_{T_{n-2}}L_{n-2})\otimes_{T_{n-1}}\mathrm{Hom}_{A_n}(A_n,
L_n)
\text{.}$$ 

\bigskip

\noindent By induction on $n$ we get

\begin{equation}\label{iso}
\mathrm{Hom}_{S_1\otimes\ldots\otimes
S_n}(A_1\otimes\ldots\otimes A_n, L_1\otimes_{T_1}\ldots\otimes_{T_{n-1}} L_n) \cong
\end{equation}
$$\mathrm{Hom}_{S_1}(A_1,
L_1)\otimes_{T_1}\ldots\otimes_{T_{n-1}} \mathrm{Hom}_{S_n}(A_n, L_n)\text{.}$$

\bigskip

Let now $X_{ij}$ be modules as in \eqref{alpha}. So we have rings $R_{ij}$, $0
\leq i \leq n$, $1 \leq j \leq m_n$ with $R_{im_i} = R_{(i+1)0} = R_i$ and $X_{ij}$ is a module
$R_{ij-1}^o\otimes R_{ij} \spanarr R_{ij-1}^o\otimes R_{ij}$. In the above, set

$$T_i = R_{i1}\otimes\ldots\otimes R_{im_i-1}$$
$$S_i = R_{i1}^o\otimes R_{i1}\otimes\ldots\otimes R_{im_i-1}^o\otimes
R_{im_i-1}\text{,}$$ $$A_i = e\otimes \ldots\otimes e\text{,}$$
$$L_i = X_{i1}\otimes_{R_{i1}}\ldots\otimes_{R_{im_i-1}}X_{im_i}\text{.}$$

\noindent Then 

$$\mathrm{Hom}_{S_i}(A_i, L_i) =
X_{i0}\times_{R_{i1}}\ldots\times_{R_{im_i}}X_{im_i}\text{.}$$ 

\noindent If $X_{ij}^{R_{ij}}$ are projective then $L_i$ is projective by Lemma
\ref{l1}. If $X_{ij}^{R_{ij}^o}$ are hereditry and $R_i^o$ is right hereditry
then $\mathrm{Hom}_{S_i}(A_i, L_i)^{R_i^o}$ is hereditry by Lemma \ref{ll1} and
\ref{ll}. It follows that there is an isomorphism \eqref{iso}. The map
$\beta_{m_1+\ldots+m_n}$ is the result of application of

$$\mathrm{Hom}_{(S_1\otimes\ldots\otimes S_{n-1})}(A_1\otimes\ldots A_{n-1},
-)$$ 

\noindent to this isomorphism and hence an isomorphism itself. 

Suppose now that the $X_{ij}$ are unions of projectives both as left and right
modules and a the base rings $R_{ij}$ are hereditry. Each of $X_{ij}$ can be
written as a union of submodules which are projective both as left and
right modules. Then since $R_{ij}$ and $L_i$ are unions of projectives
and hence flat both sides in \eqref{alpha} are unions of submodules obtained by
varying the arguments in \eqref{alpha} to projective submodules. Restrictions
of \eqref{alpha} to these submodules are isomorphisms hence \eqref{alpha} is an
isomorphism itself.
\end{proof}

Algebroids and quantum
categories in $(k\operatorname{-Mod})^{\mathrm{op}}$ are the
$\times_A$-coalgebras and $\times_R$-bialgebras of Takeuchi \cite{Takeuchi}, these were later called
coalgebroids and bialgebroids respectively.

\begin{definition}
A comodule between coalgebroids is a module between algebroids in
$(k\operatorname{-Mod})^{\mathrm{op}}$.
\end{definition}

\begin{definition}
A comodule algebra between bialgebroids is a quantum module between
quantum categories in $(k\operatorname{-Mod})^{\mathrm{op}}$.
\end{definition}

\bigskip

The operation $\bullet$ is defined by the equalizer

$$M_1\bullet M_2 \to M_1\times_R M_2 \two M_1\times_RA\times_R N_M$$

\bigskip 

The class of all those modules $X : R^o\otimes R\spanarr R^o\otimes R$ where $R$
is a hereditary ring and $X_R$ and $X_{R^o}$ are unions of projectives
satisfies all the conditions of Theorem \ref{theorem}. From Theorem \ref{t1} and
Theorem \ref{t2} we get:

\begin{theorem}
Coalgebroids $(A, R)$ with $R$ a hereditary ring and $A^R$ and
$A^{R^o}$ unions of projectives, and comodules between bialgebroids
$M$ with $M^R$ and $M^{R^o}$ unions of projectives, form a bicategory.
\end{theorem} 
 
\begin{theorem}
Bialgebroids $(A, R)$ with $R$ a hereditary ring and $A^R$ and
$A^{R^o}$ unions of projectives, and comodule algebras between bialgebroids
$M$ with $M^R$ and $M^{R^o}$ unions of projectives, form a bicategory.
\end{theorem}  

\section{Modules of weak bialgebras}

Cauchy completion $\mathcal{QV}$ of a category $\mathcal{V}$ is
the category whose objects are pairs $(X, e)$, where $X$ is an object and $e :
X \to X$ is an idempotent.  A morphisms $f : (X, e) \to (Y,
e')$ in $\mathcal{V}$ is a map $f : X \to Y$ such that $e'fe = f$. Note that the
identity on $(X, e)$ is $e$. Idempotents split in $\mathcal{QV}$.  

We will assume that idempotents split in $\mathcal{V}$ itself. In this case
$\mathcal{QV}$ is equivalent to $\mathcal{V}$. Using this equivalence we will
sometimes identify an object $(X, e)$ of $\mathcal{QV}$ with its splitting in
$\mathcal{V}$. 

A parallel pair of morphisms $f_1, f_2 : X \to Y$ in $\mathcal{V}$ is called
cosplit if there exists an arrow $d : Y \to X$ such that 

$$df_1 = 1 \qquad f_1df_2 = f_2df_2$$

\bigskip

\noindent The map $df_2$ is an idempotent whose splitting
provides a cosplit coequalizer for the pair $f_1$, $f_2$.

A Frobenius monoid in a monoidal category $\mathcal{V}$ is an
object $C$ with a monoid and a comonoid structures on it related by 

\begin{center}
\includegraphics[scale=0.3]{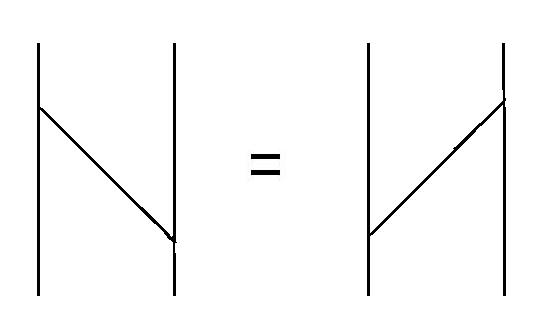}
\end{center}
 
\noindent A Frobenius monoid is separable if additionally

\begin{center}
\includegraphics[scale=0.3]{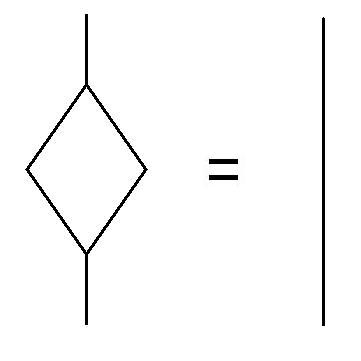}
\end{center}

\noindent Every separable Frobenius monoid $C$ is self-dual in $\mathcal{V}$
with unit and counit: 

\begin{center}
\includegraphics[scale=0.3]{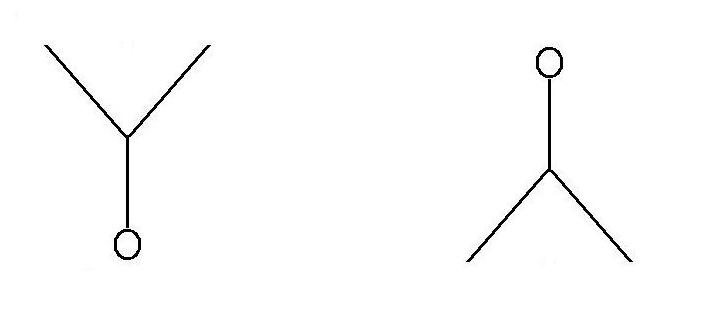}
\end{center}

Suppose that $X$ is a right $C$-comodule and $Y$ is a left $C$-comodule. If
$C$ is separable Frobenius, then the pair of morphisms 

$$X\otimes Y \two^{\delta_r^X}_{\delta_l^Y} X\otimes C\otimes Y$$

\noindent is cosplit by
$(1\otimes\mu\otimes1)(1\otimes1\otimes\delta_l)$. The induced
idempotent $a$ on $X\otimes Y$ in the string notation is

\begin{center}
\includegraphics[scale=0.3]{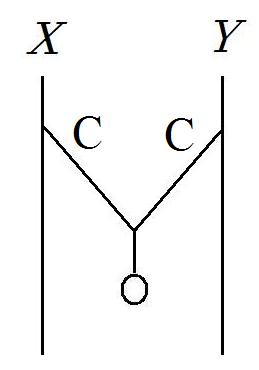}
\end{center}

\noindent So if $C$ is a separable Frobenius, then we have

\begin{equation}\label{a}
X\otimes_C Y = (X\otimes Y, a)\text{.}
\end{equation}

\bigskip
 
Suppose that $X$ is a right $C$-module and $Y$ is a left $C$-module. 
$X\otimes_C Y$ is defined by a coequalizer of the pair

\begin{equation}\label{mp}
X\otimes C\otimes Y \two X\otimes Y\text{.}
\end{equation}

\bigskip

\noindent Very much like the case of comodules if $C$ is a separable Frobenius monoid then 

\begin{equation}\label{m}
X\otimes Y = (X\otimes Y, b)\text{,}
\end{equation}

\bigskip

\noindent where $b$ is the idempotent on $X\otimes Y$: 

\begin{center}  
\includegraphics[scale=0.3]{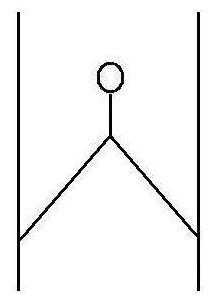}
\end{center}

Henceforce $C$ will be a separable Frobenius monoid.

The coaction of a comodule $$X : C^o\otimes C' \spanarr C^o\otimes
C'$$ in string notation is

\begin{center}
\includegraphics[scale=0.3]{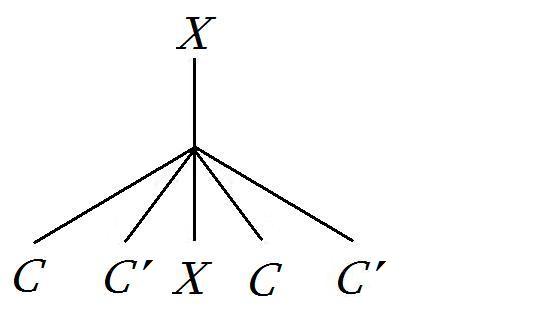}
\end{center}

Since such a comodule is regarded as a left $C$-comodule and a right
$C$-comodule using the right $C^o\otimes C$ coaction on it, the tensor product
$X_1\otimes_C X_2$ over $C$ will be $(X_1\otimes X_2, a)$, with $a$ being the
idempotent

\begin{center}
\includegraphics[scale=0.3]{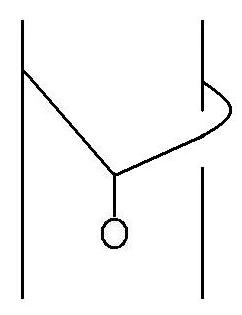}
\end{center}

\noindent Note that here we have taken a free hand with string notation. In the
above diagram it is not clear that on the left string we are using the right
$C$-coaction and on the right string we are using the right $C^o$-coaction.
However this should be clear from the context. The same occurs below. 
  
Since $C$ is selfdual in $\mathcal{V}$, $\mathrm{coHom}(C, X)$ exists
for every $X$ and is given by

$$\mathrm{coHom}(C, X) = C\otimes X\text{,}$$ 

\bigskip

\noindent with coevaluation 

\begin{center}
\includegraphics[scale=0.3]{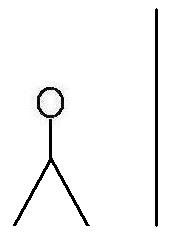}
\end{center}

The diagram $\eqref{coeqpair}$ becomes 

\bigskip   

\begin{equation}\label{ftn} 
(C\otimes X_1\otimes X_2, C\otimes a)
\two^{(\delta_l^C)^\ast}_{(\delta_l^{C^o})^\ast} (X_1\otimes X_2, a)\text{.}
\end{equation}

\bigskip 

\noindent The maps $(\delta_l^C)^\ast, (\delta_l^{C^o})^\ast :
C\otimes X_1\otimes X_2 \to X_1\otimes X_2$ are

\begin{center}
\includegraphics[scale=0.3]{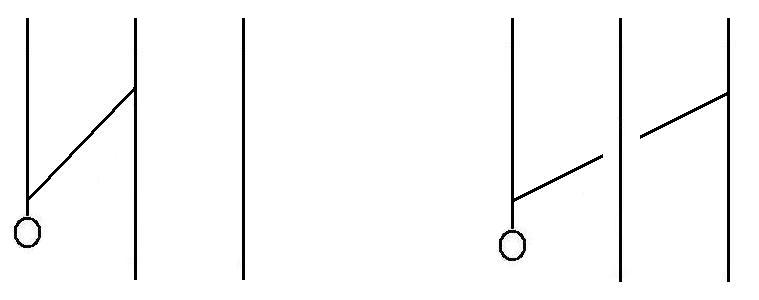}
\end{center}

\bigskip

Up to precomposing with the isomorphism $c\otimes X_2$ the
pair $(\delta_l^C)^\ast, (\delta_l^{C^o})^\ast$ is an instance of the
pair \eqref{mp} in $\mathcal{V}$. So the coequalizer is computed by
\eqref{m}. Taking the coequalizer of \eqref{ftn} we get

$$T_2(X_1, X_2) = (X_1\otimes X_2, d_2)\text{,}$$

\bigskip 

\noindent where the idempotent $d_2$ is 

\begin{center}
\includegraphics[scale=0.3]{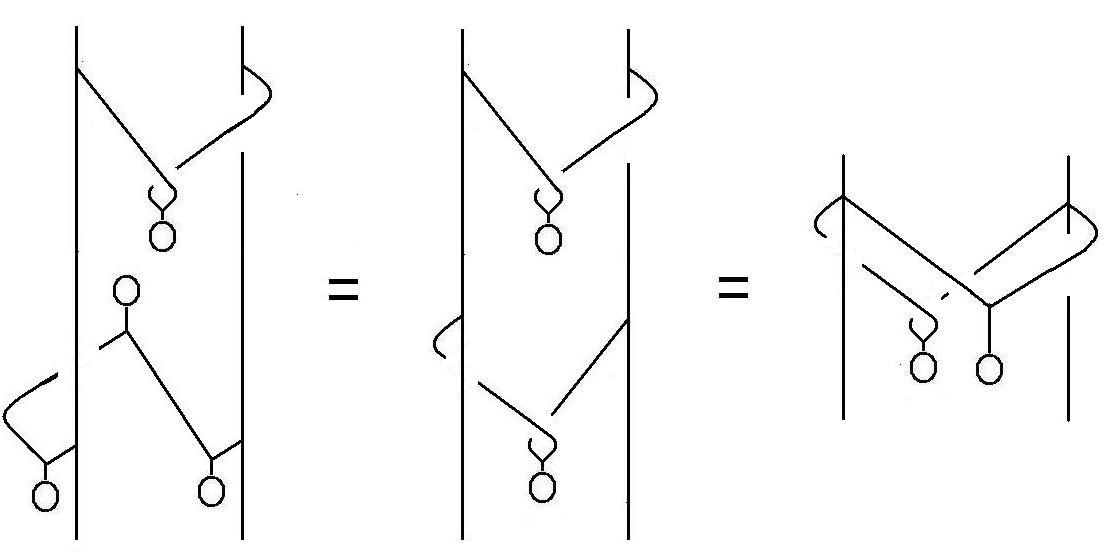}
\end{center}

\noindent $T_2(X_1, X_2)$ is a comodule $C^o_0\otimes C_2'\spanarr C^o_0\otimes
C_2'$ with coaction $(X_1\otimes X_2, d) \to (C_0\otimes X_1\otimes X_2\otimes C_2,
1\otimes d_2\otimes 1)$:
 
\begin{center}
\includegraphics[scale=0.3]{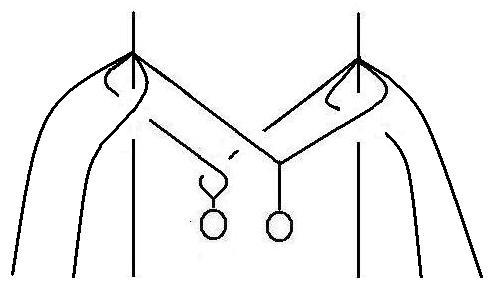}
\end{center}

Generally we have:

\begin{theorem} If the base comonoids are separable Frobenius monoids, then for
$n \geq 2$

$$T_n(X_1\ldots X_n) = (X_1\otimes\ldots \otimes X_n, d_n)\text{,}$$ 

\bigskip

\noindent where $d_n$ is the idempotent

\begin{center}
\includegraphics[scale=0.3]{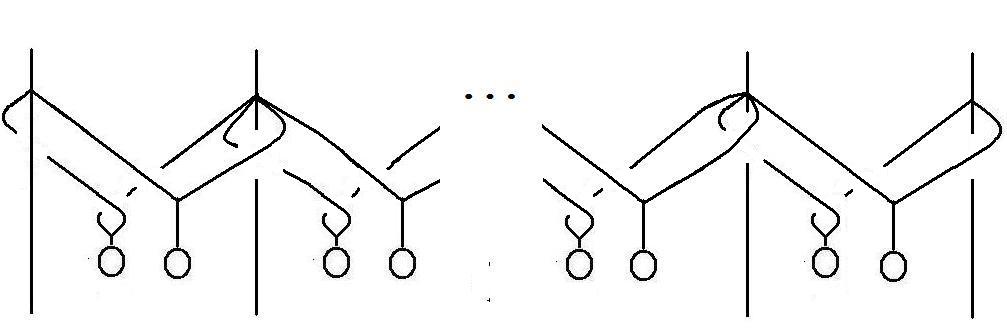} 
\end{center}
 
\end{theorem}

We can see that if the base comonoids are Frobenius
separable, then for a partition $\xi$ not envolving parts of zero length the
map $\beta_\xi$ \eqref{alpha} is an isomorphism.

$T_C() = C\otimes C$ with left and right $C^o\otimes C$-coactions: 

\begin{center}
\includegraphics[scale=0.3]{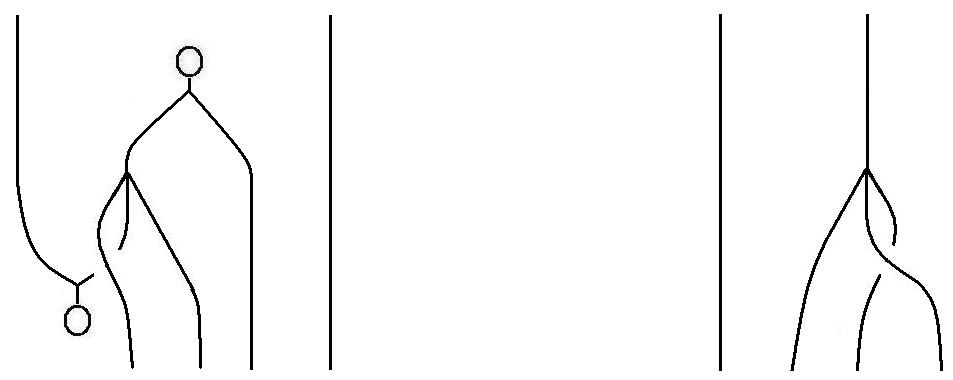} 
\end{center}

The next
proposition asserts that the class of all comodules $C^o\otimes C' \spanarr C^o\otimes C'$ with $C$ a Frobenius monoid satisfies
 condition 3 in Theorem $\ref{t1}$. The remaining conditions are obviously
satisfied.

\begin{prop}
Suppose that $X$ is a right $C$-comodule. If $C$ is a separable Frobenius
monoid and $X\otimes - : \mathcal{V} \to \mathcal{V}$ preserves reflexive
coequalizers, then the functor

$$X\otimes_C - : \mathrm{Comod}\mathcal{V}(C, I) \to \mathcal{V}$$

\bigskip 

\noindent preserves reflexive coequalizers.
\end{prop}

\begin{proof}
Let 

$$X_1 \two X_2 \to X_3$$

\noindent be a reflexive coequalizer in $\mathrm{Comod}\mathcal{V}(C, I)$. It
is computed as a coequalizer in $\mathcal{V}$. Since $X\otimes -$ preserves
reflexive coequalizers

$$X\otimes X_1 \two X\otimes X_2 \to X\otimes X_3$$

\noindent also is a coequalizer. Further, we have a coequalizer in
$\mathcal{QV}$:

$$(X\otimes X_1, a_1) \two (X\otimes X_2, a_2) \to (X\otimes X_3, a_3)$$

\noindent where $a_i$ are idempotents as in \eqref{a}. By splitting these
idempotents we prove that

$$X\otimes_C X_1 \two X\otimes_C X_2 \to X\otimes_C X_3$$

\bigskip   

\noindent is a coequalizer in $\mathcal{V}$.
\end{proof}

In \cite{PS} it was shown that a quantum category with a separable Frobenius
base monoid is the same as a weak bialgebra in $\mathcal{V}$, which is an object
$A$ with a comonoid and monoid structures on it related in a certain way. 

\begin{definition}
A module between weak bialgebras is a module between quantum categories with a
separable Frobenius base monoid. 
\end{definition} 

\begin{definition}
A module comonoid between weak bialgebras is a quantum module between quantum categories with separable Frobenius base comonoid. 
\end{definition}


\begin{thebibliography}{widest-label}

\bibitem{B} J. B\'enabou, Introduction to bicategories, Lecture Notes in
Mathematics (Springer-Verlag, Berlin) 47 (1967), 1--77. 
\bibitem{DMS} B. Day, P. McCrudden, R. Street, Dualizations and antipodes,
Applied Categorical Structures 11 (2003) 219-227. 
\bibitem{DS} B. Day and R. Street, Quantum categories, star autonomy, and
quantum groupoids, Fields Institute Communications (American Math. Soc.) 43
(2004), 187--226. 
\bibitem{KW} A. Kock, G. C. Wraith, Elementary toposes. Lecture Notes Series, No. 30. Matematisk Institut, Aarhus Universitet, Aarhus, 1971. i+118 pp.
\bibitem{PS} Pastro, C.; Street, R. Weak Hopf monoids in braided monoidal categories. Algebra Number
Theory 3 (2009), no. 2, 149--207.
\bibitem{FTM1} R. Street, The formal theory of monads, J. Pure
Appl. Algebra 2 (1972) 149--168. 
\bibitem{path} R. Street, Quantum Groups: a
path to current algebra,  Australian Mathematical Society Lecture Series 19, Cambridge University press, 2007. 
\bibitem{Sweedler} M.E. Sweedler, Groups of simple algebras, Inst. Hautes
\'etudes Sci. Publ. Math. No. 44 (1974) 79--189. \bibitem{Takeuchi} M. Takeuchi,
Groups of algebras over $A\otimes\bar{A}$, J. Math. Soc. Japan 29 (1977), 459--492
\end{thebibliography}
\end{document}